\documentclass[a4paper,12pt]{article}
\usepackage{pifont}
\usepackage{amsmath,amsthm,amssymb,amsfonts}
\usepackage{longtable}
\usepackage{multirow}
\usepackage{leftidx}
\usepackage[noend]{algpseudocode}
\usepackage{algorithmicx,algorithm}
\usepackage{bbm,bm}
\usepackage{latexsym}
\usepackage{mathrsfs}
\usepackage{threeparttable}
\usepackage{tabularx}
\usepackage{booktabs}
\usepackage{multicol,graphics}
\usepackage{subfigure}
\usepackage[dvips]{graphicx}
\usepackage{rotating}
\usepackage{epstopdf}
\usepackage{color}
\usepackage{epsfig}
\usepackage{epstopdf}
\usepackage[dvips]{graphicx}
\usepackage{multicol,graphics}
\usepackage{epsfig}
\usepackage{color}
\usepackage{subfigure}
\usepackage{indentfirst}
\usepackage{geometry}
\usepackage{caption}
\usepackage{lineno}
\usepackage{enumerate,mdwlist}
\usepackage[numbers,sort&compress]{natbib}  
\usepackage[utf8]{inputenc}
\usepackage{float}
\theoremstyle{definition}
\newtheorem{The}{Theorem}[section]    
\newtheorem{Lem}[The]{Lemma}
\newtheorem{Cor}[The]{Corollary}
\newtheorem{Pro}[The]{Proposition}

\newtheorem{Cla}{Claim}
\newtheorem{Prob}[The]{Problem}

\newtheorem{Cas}{Case}

\theoremstyle{remark}

\newtheorem{Not}[The]{Note}

\theoremstyle{definition}

\numberwithin{equation}{section}

\begin{document}
\title{Cubic bricks  that every $b$-invariant edge is forcing
\footnote{This work is supported by NSFC (Grant Nos. 12271229 and 12271235).}}
\author{Yaxian Zhang$^1$, Fuliang Lu$^2$ and Heping Zhang$^1$\thanks{Corresponding author.} }
\date{
   {\small 1. School of Mathematics and Statistics, Lanzhou University, Lanzhou, Gansu 730000, P.R. China}\\
   {\small 2. School of Mathematics and Statistics, Minnan Normal University, Zhangzhou, Fujian 363000, P.R. China}\\
   {\small E-mails:\ yxzhang2016@lzu.edu.cn, flianglu@163.com, zhanghp@lzu.edu.cn}
}

\maketitle
\begin{abstract}
   A connected graph $G$ is matching covered if every edge lies in some perfect matching of $G$. Lov\'asz  proved that every matching covered graph $G$ can be uniquely
   decomposed into a list of bricks (nonbipartite) and braces (bipartite) up to multiple edges. Denote by
   $b(G)$  the number of bricks of $G$.
   An edge $e$ of $G$ is removable if $G-e$ is also matching covered, and solitary (or forcing)
   if after the removal of the two end vertices of $e$, the left graph has a unique perfect matching.
   Furthermore, a removable edge $e$ of a brick $G$ is $b$-invariant if $b(G-e)=1$.

   Lucchesi and Murty proposed a problem of characterizing bricks,
   distinct from $K_4$, $\overline{C_6}$ and the Petersen graph,
   in which every $b$-invariant edge is forcing.
   We answer the problem for cubic bricks by showing that there are exactly
   ten cubic bricks, including $K_4$, $\overline{C_6}$ and the Petersen graph,
   every $b$-invariant edge of which is forcing.
    \medskip
\end{abstract}
\section{Introduction}
All graphs in this paper are finite and contain no loops (multiple edges are allowed).
An edge subset of a graph is a \emph{perfect matching} if
every vertex is incident with exactly one edge of the subset.
A connected graph is \emph{matching covered} if it contains at least one edge,
and each edge belongs to some perfect matching.
For the terminology on matching covered graphs,
we follow Lov\'asz and Plummer \cite{Lovasz1986}.

Let $G$ be a connected graph with vertex set $V(G)$ and edge set $E(G)$.
For any nonempty proper subset $X\subset V(G)$,
we call $\nabla(X)$ an \emph{edge cut} of $G$, which is the set of edges with one end vertex
in $X$, the other in $\overline{X}=V(G)-X$.
Specially, if $X$ or $\overline{X}$ contains only one vertex, then the edge cut
$\nabla(X)$ is called \emph{trivial}; otherwise, \emph{nontrivial}.
An edge cut of $G$ is a \emph{tight cut},
if it intersects each perfect matching of $G$ exactly at one edge.
A matching covered graph without nontrivial tight cuts is a \emph{brick},
if it is nonbipartite; a \emph{brace}, otherwise.

A graph $G$, with four or more vertices, is \emph{bicritical} if $G-u-v$
has at least one perfect matching for every pair of distinct vertices $u$ and $v$.
Edmonds et al. \cite{Edmonds1982} showed that
a graph $G$ is a brick if and only if it is 3-connected bicritical
(also see Lov\'asz \cite{Lovasz1987},
Szigeti \cite{Szigeti2002} and de Carvalho et al. \cite{Carvalho2018}).
 Every matching
covered graph $G$ can be decomposed into a unique list of bricks and braces up to multiple edges by a procedure
called the tight cut decomposition (see Lov\'asz \cite{Lovasz1987}).
Let $b(G)$ denote the number of bricks in such list.

An edge $e$ of a matching
covered graph $G$ is removable if $G-e$ is also matching covered, and is $b$-invariant if $b(G-e)=b(G)$.
Confirming a conjecture of Lov\'asz, de Carvalho et al.  obtained the following result.

\begin{The}
   \rm \cite{Carvalho2002}
   \label{ex-b}
   Every brick, other than $K_4$, $\overline{C_6}$ and the Petersen graph, has at least
   one $b$-invariant edge.
\end{The}

De Carvalho et al. \cite{Carvalho2002II} also showed that
the bicorn (the graph $G_2$ in Fig. \ref{f4}) is the only brick with exactly one $b$-invariant edge.
Kothari et al. \cite{KCLL20} showed that every essentially 4-edge-connected cubic nonnear-bipartite brick $G$, distinct from the Petersen graph, has at least $|V(G)|$ $b$-invariant edges.
Moreover, they conjectured every essentially 4-edge-connected cubic near-bipartite brick $G$, distinct from $K_4$, has at least $|V(G)|/2$ $b$-invariant edges;
Lu et al. \cite{LXY19} confirmed this conjecture.
A brick is {\em solid} if $G-(V(C_1)\cup V(C_2))$ has no perfect matching for any two vertex disjoint odd cycles $C_1$ and $C_2$.
De Carvalho et al. \cite{CLM12} proved that every solid brick, distinct from $K_4$, has at least $\frac{|V(G)|}{2}$ $b$-invariant edges.

Lucchesi and Murty \cite{Lucchesi2022} used $b$-invariant edges to characterize extremal
matching covered graphs and proposed solitary edge.
An edge of a graph is \emph{solitary} if it lies in precisely one perfect matching.
In fact the solitary edge appeared in benzenoid hydrocarons of theoretical chemistry under name ``forcing edge''.
General forcing concept, under another name innate degree of freedom due to
Klein and Randi\'c \cite{Klein1987}, were proposed by Harary et al.
\cite{Harary1991}.
For hexagonal systems, Zhang et al. \cite{zhang1988} completely determined the structure
of a hexagonal system whose resonance graph has a 1-degree vertex and exactly two 1-degree vertices respectively.
As a consequence, Zhang and Li \cite{zhang1995} determined
the hexagonal systems with a forcing edge.
Hansen and Zheng \cite{hansen1994} independently obtained the same result.
For polyominoes, Sun et al. \cite{sun2024} recently
characterized those graphs whose resonance graphs have a 1-degree
vertex, and as a consequence, determined all such graphs with a forcing edge.
For 3-connected cubic graphs, Wu et al. \cite{Ye2016} showed that the graph with a
forcing edge can be generated from $K_4$ via replacing a vertex by a triangle repeatedly.

A matching covered graph $G$ is \emph{extremal}
if the number of perfect matchings of $G$ is equal to the dimension of the
lattice spanned by the set of incidence vectors of perfect matchings of $G$.
De Carvalho et al. \cite{Carvalho2004} proved that a matching covered graph $G$
is extremal if and only if (i) every $b$-invariant edge of $G$ is solitary and (ii) for any $b$-invariant edge of $G$,
the graph $G-e$ is extremal.
More generally, Lucchesi and Murty proposed the following problem, see Unsolved Problems 1. in \cite{Lucchesi2022}.

\begin{Prob}\rm \cite{Lucchesi2022}
Characterize bricks, distinct from $K_4$, $\overline{C_6}$ and the Petersen graph, in which every
$b$-invariant edge is solitary.
\end{Prob}

Obviously, $K_4$, $\overline{C_6}$ and the Petersen graph are cubic bricks in which every $b$-invariant edge of $G$ is solitary,
since they contain no $b$-invariant edges \cite{Carvalho2002}.
In this paper, we consider other cubic bricks.
Based on a result in \cite{Ye2016},
a cubic brick with a forcing edge can be generated from $K_4$ via replacing a vertex by a triangle repeatedly.
According to study changes of forcing edges and $b$-invariant edges under the above operation,
we get the main result as follows.

\begin{The}
   \label{th1}
   Let $G$ be a cubic brick other than $K_4$, $\overline{C_6}$ and the Petersen graph. All $b$-invariant edges in $G$
   are forcing edges if and only if $G$ is one of the graphs in $\{G_i,2\leq i\leq 8\}$
   (see Fig. \ref{f4}).
\end{The}

   \rm
It should be noted that $\{G_i,2\leq i\leq 8\}$, $K_4$, $\overline{C_6}$,
the Petersen graph and the theta graph (the cubic bipartite graph with precisely two vertices)
form the set of all extremal cubic matching covered graphs \cite{Carvalho2004}.
Thus we can directly get the following result.
\begin{Cor}
   A cubic brick $G$  is extremal if and only if every $b$-invariant edge of $G$ is solitary.
\end{Cor}

\section{ Preliminaries \label{PRE}}
An edge in a connected graph $G$ is a \emph{bridge},
if its deletion results in a disconnected graph.
For a forcing edge $e$ of a graph $G$, $G-V(e)$ has a unique perfect matching.
The following theorem is helpful for determining whether an edge is a forcing edge.
\begin{The}
   \rm \cite{Kotzig1959}
   \label{BE}
   If $G$ is a connected graph with a unique perfect matching,
   then $G$ has a bridge belonging to the perfect matching.
\end{The}

Let $X$ be a nonempty proper vertex subset of  a graph $G$.
Denote by $N_G(X)$ the subset of $V(G)\setminus X$,
every vertex in it has a neighbour in $X$;
if $X$ contains only one vertex $x$, then we simply write $N_G(X)$  as $N_G(x)$ (or simply $N(X)$ or $N(x)$ if no confusion occurs).

We follow the definition of $Y \rightarrow \triangle$-operations in
\cite{Ye2016}.
Let $x$ be a vertex of degree 3 in a graph $G$ with $N(x)=\{y_1,y_2,y_3\}$.
A \emph{$Y \rightarrow \triangle$-operation} of $G$ on $x$, denoted by $G^{\triangle}(x)$ or simply  $G^{\triangle}$,  means to replace
$x$ by a triangle $x_1x_2x_3x_1$ and join  $x_i$ to $y_i$ for $i=1,2,3$
(see Fig. \ref{f1}).
We call $x_1x_2x_3x_1$ the replacement-triangle of $G^{\triangle}$ corresponding to the vertex $x$; and
$x_iy_i$ the edge of $G^{\triangle}$ corresponding to the edge $xy_i$ in $G$.

\begin{figure}[ht]
   \centering
   \includegraphics[scale=0.35]{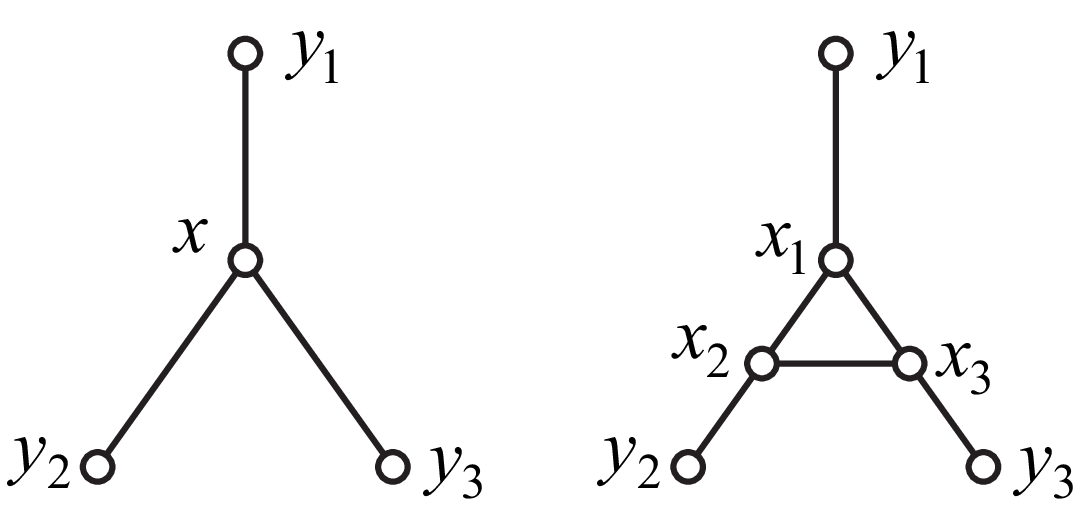}
   \caption{\label{f1} Performing a $Y \rightarrow \triangle$-operation on $x$.}
\end{figure}

 Wu et al. \cite{Ye2016} presented a theorem on
cubic graphs with forcing edges.

\begin{The}
   \rm \cite{Ye2016}
   \label{forcing-edges}
   Let $G$ be a 3-connected cubic graph.
   If $G$ has a forcing edge, then $G$ is generated from $K_4$ via a series of
   $Y \rightarrow \triangle$-operations.
\end{The}
De Carvalho et al. showed that the property of ``cubic brick'' is heritable under
$Y \rightarrow \triangle$-operations.
\begin{The} \rm \cite{Carvalho2004}
   \label{brick-expand}
   Let $G$ be a cubic brick.
   Then $G^{\triangle}$ is also a cubic brick.
\end{The}

We call a graph $H$ is a \emph{base} of a graph $G$ if
$G$ can be obtained from $H$ via a series of $Y \rightarrow \triangle$-operations.
Note that $|V(H)|<|V(G)|$.
By Theorem \ref{brick-expand}, every graph that contains $K_4$ as a base is a cubic brick.
Let $X$ be a vertex subset of $G$.
We denote the graph, obtained from $G$ by contracting  $X$ to a single vertex $x$, by
 $G/X\rightarrow x$, or simply  $G/X$.
If $H$ is a subgraph of $G$, then let $G/V(H)=G/H$. Moreover, we have the following proposition.
\begin{Pro}
   \label{cor1}
   If $G\neq K_4$ and $G$ contains $K_4$ as a base, then for any triangle $T$ in $G$,
   $G/T$ is a cubic brick.
\end{Pro}
\begin{proof}
   We apply induction on the order of $G$.
   Since $K_4$ is a base of $G$, $|V(G)|\geq 6$.
   If $|V(G)|=6$, then $G\cong \overline{C_6}$ and $G/T\cong K_4$, hence the result holds.
   Next we consider $|V(G)|\geq 8$.
   Let $T'$ be a triangle of $G$ such that $G/T'$ contains $K_4$ as a base.
   If $T'=T$, then $G/T'=G/T$.
   By Theorem \ref{brick-expand}, $G/T'$ is a cubic brick and we are done.
   If $T'\neq T$, then we can deduce $V(T')\cap V(T)=\emptyset$
   by the fact that $G$ is cubic and 3-connected,
   which implies that $(G/T')/T=(G/T)/T'$.
   By induction, $(G/T')/T$ is a cubic brick.
   Since $(G/T)/T'$ is a base of $G/T$, $G/T$ is a cubic brick by Theorem \ref{brick-expand}.
\end{proof}

Next we will show that the property of ``nonforcing'' is also heritable under
$Y \rightarrow \triangle$-operations.

\begin{Lem}
   \label{lem-f}
   Let $G$ be a cubic brick with a nonforcing edge $e$.
   Then the edge $e'$ in $G^{\triangle}$ corresponding to $e$ is also
   not a forcing edge of $G^{\triangle}$.
\end{Lem}
\begin{proof}
   Recall that $G^{\triangle}$ is obtained from $G$ by replacing $x$ with a triangle
   $T=x_1x_2x_3x_1$.
   Let $N_{G}(x)=\{y_1,y_2,y_3\}$. Then
   $e_i'=x_iy_i$ is an edge of $G^{\triangle}$ corresponding to $e_i=xy_i$ for $i=1,2,3$.
   Since $e$ is not a forcing edge of $G$,
   $G$ has at least two different perfect matchings
   $M_1$ and $M_2$ containing $e$.
   We can assume that $e_1\in M_1$ and $\{e_s\}=M_2\cap \{e_1,e_2,e_3\}$, where $s\in \{1,2,3\}$.
   Then $(M_1\setminus \{e_1\})\cup \{e_1',x_2x_3\}$
   and $(M_2\setminus \{e_s\})\cup \{e_s',E(T)\setminus \{x_s\}\}$ are two
   different perfect matchings of $G^{\triangle}$ containing $e'$.
   So $e'$ is not a forcing edge of $G^{\triangle}$.
\end{proof}

For a forcing edge of a cubic brick, if one end of the edge is replaced by a triangle, we have the following lemma.

\begin{Lem}
   \label{cla2}
   If $e=uv$ is a forcing edge of a cubic brick $G$,
   then the edge $e'$ of $G^{\triangle}(u)$ corresponding to $e$ is not a forcing edge.
\end{Lem}
\begin{proof}
   Since $G-u-v$ has a unique perfect matching $M$,
   $G-u-v$ contains at least one bridge by Theorem \ref{BE}.
   Contracting every maximal 2-edge-connected subgraph in $G-u-v$ to a vertex respectively,
   the result graph is a tree, say $H$.
   Let $G'$ and $G''$ be two subgraphs of $G$ that  corresponds to
  two leaves of $H$.
  Since $G$ is 3-connected,
  each of $u$ and $v$ has one neighbour in $G'$ and one neighbour in $G''$ respectively.
  As $G$ is cubic and $uv\in E(G)$,
  $|\nabla(\{u,v\})|=4$. Then
  $\nabla(\{u,v\})$ is exactly the set of edges between $ V(G')\cup V(G'')$ and $\{u,v\}$. Therefore,
  $H$ has no vertices of degree two by $G$ is 3-connected again.
  So $H$ is $K_2$, i.e. $G-u-v$ has exactly one bridge.

  \begin{figure}[ht]
   \centering
   \includegraphics[scale=0.3]{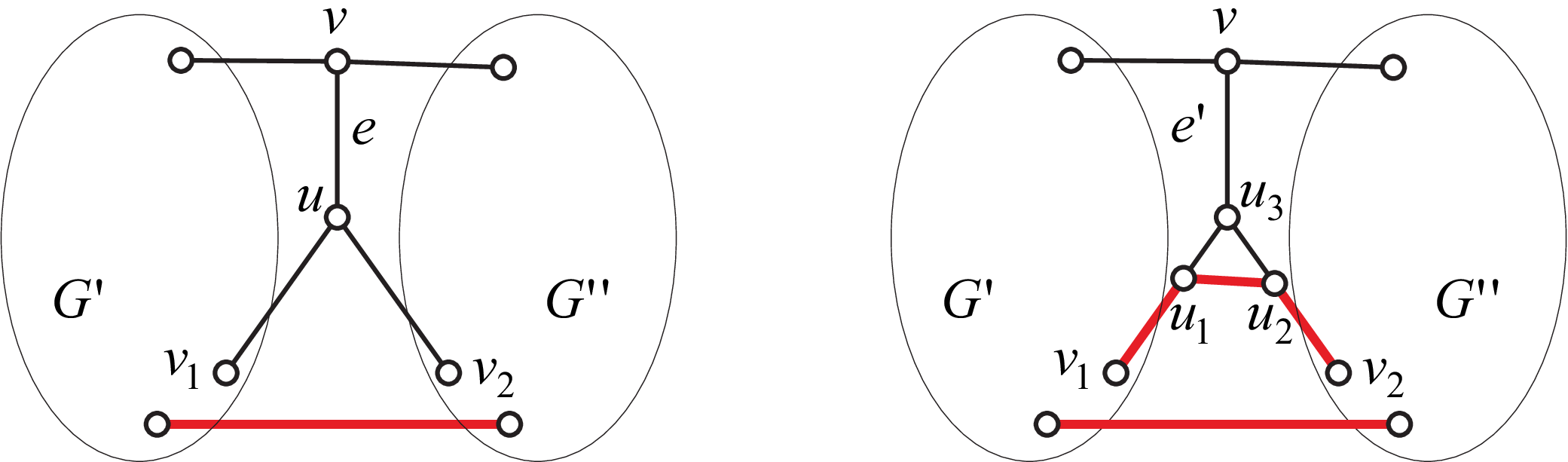}
   \caption{\label{f1-1} Illustration for the proof of Lemma \ref{cla2}.}
\end{figure}
   Note that
   for any pair of vertices in $G'$ or $G''$,
   there are two edge-disjoint paths connecting them in $G'$ or $G''$ as they are 2-edge-connected.
   Let $T=u_1u_2u_3u_1$ be the triangle in $G^{\triangle}(u)$ corresponding to $u$ and
   $u_3v$ be the edge in $G^{\triangle}(u)$ corresponding to $uv$, i.e. $e'=u_3v$.
   Since $u$ has two neighbours $v_1$ and $v_2$ distinct from $v$, we can assume
   $v_1\in V(G')$ and $v_2\in V(G'')$ such that
   $\{u_1v_1, u_2v_2\}\subset G^{\triangle}(u)$.
   Then $G^{\triangle}(u)-u_3-v$ can be obtained from $G-u-v$ by adding a   path
   $v_1u_1u_2v_2$  connecting
   $G'$ and $G''$.
   So for any pair of vertices in $G^{\triangle}(u)-u_3-v$,
   there are two edge-disjoint paths between them.
   Thus $G^{\triangle}(u)-u_3-v$ has no bridges.
   By Theorem \ref{BE}, $e'=u_3v$ is not a forcing edge of $G^{\triangle}(u)$.
\end{proof}

By Lemmas \ref{lem-f} and \ref{cla2},  we have $Y \rightarrow \triangle$-operations keep
the number of forcing edges from increasing. More exactly, we have the following corollary.
\begin{Cor}
   \label{Y-Op}
   Let $G$ be a  cubic brick.
   Then for any forcing edge $e$ of $G^{\triangle}$,
   there is exactly one forcing edge of $G$ corresponding to $e$.
   Therefore, the number of forcing edges of $G^{\triangle}$ is no more
   than the one of $G$.
\end{Cor}
\begin{proof}
   Let $T=x_1x_2x_3x_1$ be the replacement-triangle of $G^{\triangle}$ corresponding to the vertex $x$ of $G$.
   If $e\notin E(T)$, then by Lemma \ref{lem-f},
   its corresponding edge $e'$ must be a forcing edge of $G$.
   Now we   assume, without loss of generality, that $e=x_1x_2$.
   Let $N_G(x)=\{y_1,y_2,y_3\}$. Then we claim that $e'=xy_3$
   is a forcing edge of $G$.
   Let $M$ be the unique perfect matching of $G^{\triangle}$ containing $e$.
   Then $x_3y_3$ belongs to $M$.
   Further $(M\setminus \{e,x_3y_3\})\cup \{e'\}$ is the unique perfect matching of $G$ containing $e'$.
   So $e'$ is the unique forcing edge in $G$ corresponding to $e$.
   By Lemma \ref{cla2}, $x_3y_3$ is not a forcing edge of $G^{\triangle}$.
   So $G^{\triangle}$ has equal or less forcing edges than $G$.
\end{proof}

It can be checked that $K_4$ has exactly six forcing edges,
 every cubic brick contains at most six forcing edges by Theorem \ref{forcing-edges} and Corollary \ref{Y-Op}.
Let $R_0$ be the graph obtained  by replacing each vertex of $K_4$ by a triangle (see Fig. \ref{f2}). Moreover, we have the following proposition.

\begin{figure}[ht]
   \centering
   \includegraphics[scale=0.35]{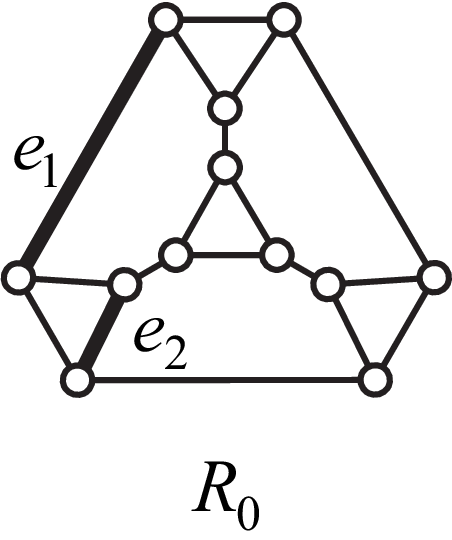}
   \caption{\label{f2} Illustration for Proposition \ref{pro1}.}
\end{figure}
\begin{Pro}
   \label{pro1}
   The graph $R_0$ contains no forcing edges. Furthermore,
   all graphs with $R_0$ as a base have no forcing edges.
\end{Pro}
\begin{proof}
   As shown in Fig. \ref{f2}, we just need to show
   that neither $e_1$ nor $e_2$ is a forcing edge of
   $R_0$ by symmetry.
   Since $R_0-V(e_2)-V(e_1)$ has no bridges, $R_0-V(e_2)-V(e_1)$ contains at least two
   perfect matchings, say $M_1$ and $M_2$, by Theorem \ref{BE}.
   Then $M_1\cup \{e_1,e_2\}$ and $M_2\cup \{e_1,e_2\}$ are two different perfect
   matchings of $R_0$ containing $e_1$ and $e_2$. So
   both $e_1$ and $e_2$ are not forcing edges of $R_0$. By Corollary \ref{Y-Op} repeatedly, those graphs with $R_0$ as a base have no forcing edges.
\end{proof}

Let $x$ be a vertex of degree 4 in a graph $G$, and
$y_1$, $y_2$, $y_3$ and $y_4$ be its neighbours, where $y_4$ may be the same vertex as $y_1$.
A \emph{generalized $Y \rightarrow \triangle$-operation} of $G$ on $x$ means to replace
$x$ by a triangle $x_1x_2x_3x_1$, join $x_i$ to $y_i$ for $i=1,2$,
and join $x_3$ to $y_3$ and $y_4$ (see Fig. \ref{Af1}). 
We call $x_1x_2x_3x_1$ the replacement-triangle corresponding to the 
vertex $x$.

\begin{figure}[ht]
   \centering
   \includegraphics[scale=0.35]{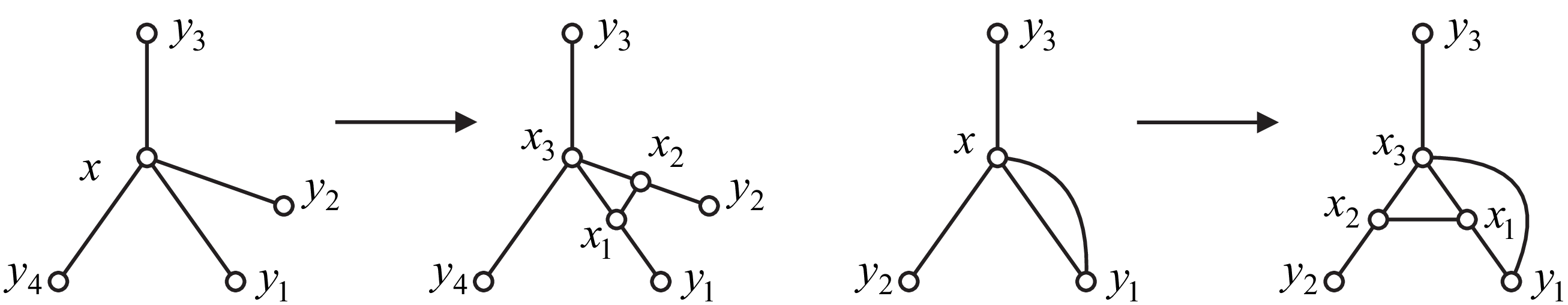}
   \caption{\label{Af1}Performing a generalized $Y \rightarrow \triangle$-operation
   on $x$.}
\end{figure}
\begin{Lem}
   \label{GY}
   Let $G$ be a brick with a vertex $x$ of degree 4.
   If $G'$ can be obtained from $G$ by a generalized $Y \rightarrow \triangle$-operation on $x$.
   Then $G'$ is also a brick.
\end{Lem}
\begin{proof}
   Let $S'$ be a vertex subset of $G'$ with size two, and $T=x_1x_2x_3x_1$ be the replacement-triangle of $G'$ corresponding to $x$.
   If $S'\cap \{x_1,x_2,x_3\}=\emptyset$, then $S'$ is a vertex subset of $G$.
   Since $G$ is 3-connected bicritical, $G-S'$ is connected and has a perfect matching $M$.
   In this case, $G'-S'$ can be obtained from $G-S'$ by replacing $x$ by $T$,
   $G'-S'$ is also connected and has a perfect matching containing exactly one edge in
   $E(T)$.

   If $|S'\cap \{x_1,x_2,x_3\}|=1$, then let $S=(S'\setminus V(T))\cup \{x\}$.
   Note that $G'-S'$ can be obtained from $G-S$ by adding
   two vertices from $T$, say $x_i$ and $x_j$, and all edges incident with them.
   As shown in Fig. \ref{Af1}, $x_ix_j\in E(G')$ and $|N(\{x_i,x_j\})\cap (V(G)\setminus S)|\geq 1$.
   Since $|S|=2$, $G-S$ is connected and has a perfect matching $M_1$,
   $G'-S'$ is also connected and has a perfect matching $M_1\cup \{x_ix_j\}$.

   If $|S'\cap \{x_1,x_2,x_3\}|=2$, then let $x_r$ be the unique vertex in $V(T)\setminus S'$.
   Surely, $x_ry_k\in E(G'-S')$ for an integer $k\in \{1, 2,3,4\}$.
   Since $G-x$ is connected, $G'-V(T)$ is also connected.
   Since $x_r$ is adjacent to $y_k$ in $G'-V(T)$,
   $G'-S'$ is also connected.
   Since $G-x-y_k$ has a perfect matching $M_2$, $M_2\cup \{x_ry_k\}$ is a perfect
   matching of $G'-S'$.
   From above discussion, we know that $G'$ is also 3-connected  and bicritical, i.e.
   $G'$ is a brick.
\end{proof}
We are going to find those cubic bricks with the property that all $b$-invariant edges are forcing.
By Theorem \ref{ex-b}, we consider those
cubic bricks with at least one forcing edge and at most six $b$-invariant edges.
To find $b$-invariant edges in a cubic brick,
we will introduce a \emph{pyramid},
which is gotten from a triangle $xyzx$ by inserting a new 2-degree vertex on
$xy$ and $xz$ respectively, and joining the two new 2-degree vertices by an edge.
We call the edge $yz$ the \emph{bottom edge} of
the pyramid.
For example, the subgraph induced by $\{u_0,u_1,\ldots, u_4\}$ in Fig. \ref{f3} is a
pyramid, where $u_3u_4$ is the bottom edge.

\begin{figure}[ht]
   \centering
   \includegraphics[scale=0.25]{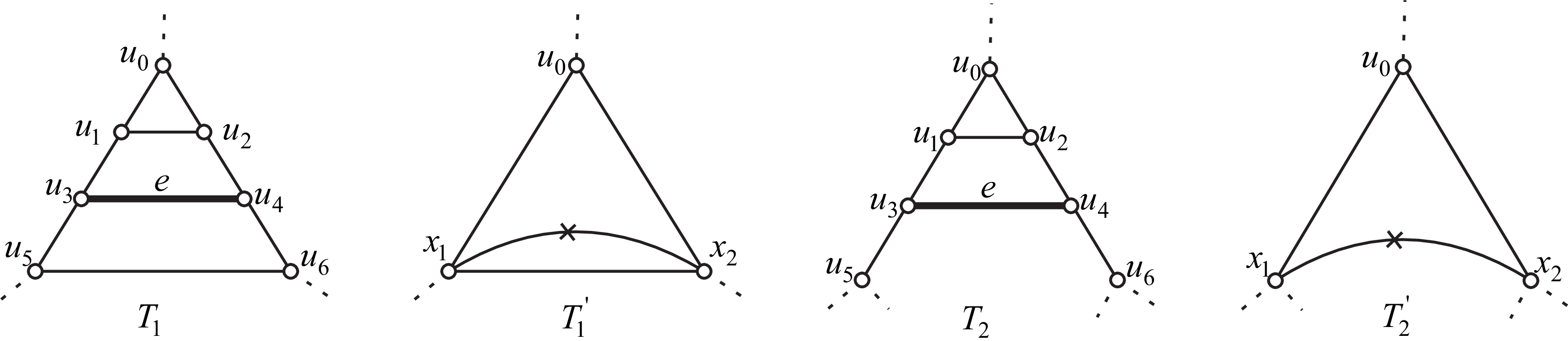}
   \caption{\label{f3}Illustration for Lemma \ref{lem4}: $T_i'=((T_i-u_3u_4)/X_1\rightarrow x_1)/X_2\rightarrow x_2$, where $X_1=\{u_1,u_3,u_5\}$, $X_2=\{u_2,u_4,u_6\}$ and  $i=1,2$.}
\end{figure}
\begin{Lem}
   \label{lem4}
   Assume that $G$ contains $K_4$ as a base,  $|V(G)|\geq 8$ and
   $G$ contains a subgraph $H$ that is isomorphic to a pyramid.
   Then the bottom edge $e$ of $H$ is a $b$-invariant edge of $G$.
   Further, if $G^*$ contains $G$ as a base and both end vertices of $e$ lie in  $G^*$,
  then $e$ is also a $b$-invariant edge of $G^*$.
\end{Lem}
\begin{proof}
   As shown in Fig. \ref{f3}, we label all vertices of $H$ by $u_0,u_1,\ldots,u_4$,
   where $e=u_3u_4$ is the bottom edge of $H$.
   Obviously, the deletion of $e$ from $G$ will produce two
   2-degree vertices $u_3$ and $u_4$ in $G-e$.
   Let $u_5\in N(u_3)\setminus\{u_1,u_4\}$ and  $u_6\in N(u_4)\setminus\{u_2,u_3\}$.
   If $u_5=u_6$, then from $|V(G)|\geq 8$,
   $\{u_0,u_5\}$ is a 2-vertex cut of $G$, a contradiction.
   So $u_5\neq u_6$.
   Let $X_1=\{u_1,u_3,u_5\}$ and $X_2=\{u_2,u_4,u_6\}$. Then
   $X_1\cap X_2=\emptyset$.

   For any perfect matching $M$ of $G-e$, since $|X_1|=|X_2|=3$,
   $|M\cap (\nabla(X_i)\setminus \{e\})|$ is odd for $i=1,2$.
   Since $\nabla(X_1)\setminus \{e\}$ is covered by $u_1$ and $u_5$
   and $\nabla(X_2)\setminus \{e\}$ is covered by $u_2$ and $u_6$,
   $|M\cap \nabla(X_i)|=1$.
   So $\nabla(X_1)\setminus \{e\}$ and $\nabla(X_2)\setminus \{e\}$ are both nontrivial tight cuts of $G-e$.
   Since $(G-e)/\overline{X_1}$ and $(G-e)/\overline{X_2}$ are both isomorphic to a 4-cycle
   with two multiple edges, they are both matching covered bipartite graphs.
   Let $G'=((G-e)/X_1\rightarrow x_1)/X_2\rightarrow x_2$ (may be a multiple graph).
   Then we will show that $G'$ is a brick.

As $u_1u_2\in E(G)$, $x_1x_2\in E(G')$.
   Let $G''$ be the graph obtained from $G'$ by deleting exactly one edge joining
   $x_1$ and $x_2$ (if $u_5u_6\in E(G)$, then there exist two edges joining $x_1$ and $x_2$ in $G'$).
   Then $G''$ can be seen as a cubic graph obtained from $G$ by twice triangle contractions, i.e.
   $G''\cong (G/\{u_0,u_1,u_2\}\rightarrow u)/\{u,u_3,u_4\}$.
   Since $G$ contains $K_4$ as a base, $G''$ is also a cubic brick by Proposition \ref{cor1}.
   Since $G'$ can be obtained from $G''$ by adding exactly one edge $x_1x_2$, $G'$ is also a brick. 
   Since $G'$ is the only brick obtained from $G-e$ by contracting two nontrivial tight cuts,
   $b(G-e)=1$, which implies that $e$ is a $b$-invariant edge of $G$.

   Next we show that $e=u_3u_4$ is also a $b$-invariant edge of $G^*$.
   Similar to the above, let $X_1'$ (resp. $X_2'$) be the vertex subset formed by
   $u_3$ (resp. $u_4$) and its two neighbours in $G^*-u_3u_4$.
   Then $\nabla(X_1')$ and $\nabla(X_2')$ are also both nontrivial tight cuts of
   $G^*-e$.
   Since $u_3$ and $u_4$ belong to $V(G^*)$,
   $(G^*-e)/\overline{X_1'}$ and $(G^*-e)/\overline{X_2'}$ are also
   matching covered bipartite graphs.

   Since $G$ is a base of $G^*$, $G-e$ is a base of $G^*-e$.
   If $X_1=X_1'$ and $X_2=X_2'$, then $((G-e)/X_1)/X_2$ is a base of $((G^*-e)/X_1')/X_2'$.
   If $X_1\neq X_1'$ or $X_2\neq X_2'$, then at least one vertex of $\{u_1,u_2,u_5,u_6\}$
   is not a vertex of $G^*$.
   In the graph $((G-e)/X_1\rightarrow x_1)/X_2\rightarrow x_2$, both $x_1$ and $x_2$ have degree four.

   \begin{figure}[ht]
      \centering
      \includegraphics[scale=0.3]{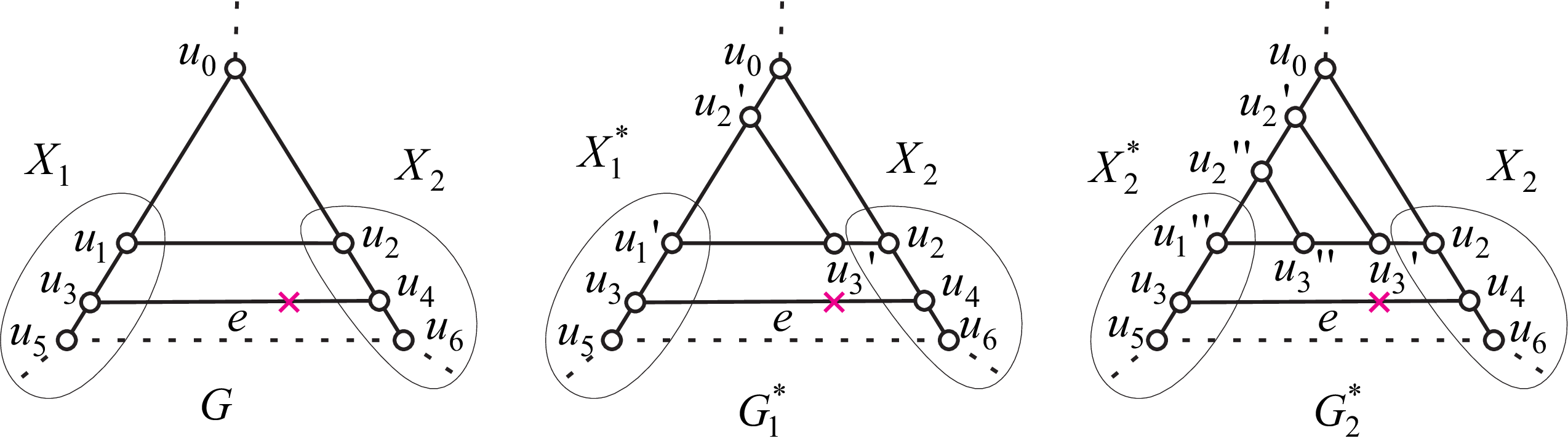}
      \caption{\label{f3-1}Illustration for $b(G^*-e)=1$.}
   \end{figure}

   Let $G_1^*=G^{\triangle}(u_1)$ and $X_1^*=\{u_1',u_3,u_5\}$,
   where the edge $u_1'u_3$ in $G_1^*$ corresponds to $u_1u_3$ in $G$.
   Then $(G_1^*-e)/X_1^*/X_2$ can be obtained from
   $((G-e)/X_1\rightarrow x_1)/X_2$ via a generalized $Y\rightarrow \triangle$-operation
   on $x_1$ (see Fig. \ref{f3-1}).
   Let $G_2^*=(G_1^*)^{\triangle}(u_1')$ and $X_2^*=\{u_1'',u_3,u_5\}$,
   where the edge $u_1''u_3$ in $G_2^*$ corresponds to $u_1'u_3$ in $G_1^*$.
   Then $(G_2^*-e)/X_2^*/X_2$ can be obtained from
   $(G_1^*-e)/X_1^*\rightarrow x_1')/X_2$ via a generalized
   $Y\rightarrow \triangle$-operation on $x_1'$.

From above discussion, we know that a  $Y\rightarrow \triangle$-operation on a vertex adjacent to $u_3$ (resp. $u_4$) corresponds to a generalized $Y\rightarrow \triangle$-operation on a 
vertex of degree 4; a  $Y\rightarrow \triangle$-operation on a vertex not adjacent to $u_3$ (resp. $u_4$) corresponds to a $Y\rightarrow \triangle$-operation on a 
vertex of degree 3.
So  $((G^*-e)/X_1')/X_2'$ can be obtained from  $((G-e)/X_1)/X_2$ via
$Y\rightarrow \triangle$-operations on its vertices of degree 3, or via generalized $Y\rightarrow \triangle$-operations on its 
vertices of degree 4 repeatedly.
   Since $G'=((G-e)/X_1)/X_2$ is a brick,
   $((G^*-e)/X_1')/X_2'$ is also a brick by Theorem \ref{brick-expand} and Lemma \ref{GY}.
   So $e$ is a $b$-invariant edge of $G^*$.
\end{proof}

By Lemma \ref{lem-f} and Lemma \ref{lem4}, we can directly
get the following result.

\begin{Cor}
   \label{bf}
   Let $G$ be a graph containing $K_4$ as a base  and $|V(G)|\geq 8$, $H$ be a pyramid of $G$
   with the bottom edge $e$. Assume that a graph $G^*$ contains $G$ as a base and the two end vertices of
   $e$ lie in $G^*$.
   If $e$ is not a forcing edge of $G$, then $e$ is a $b$-invariant but not forcing edge of $G^*$.
\end{Cor}

\begin{figure}[h]
   \centering
   \includegraphics[scale=0.35]{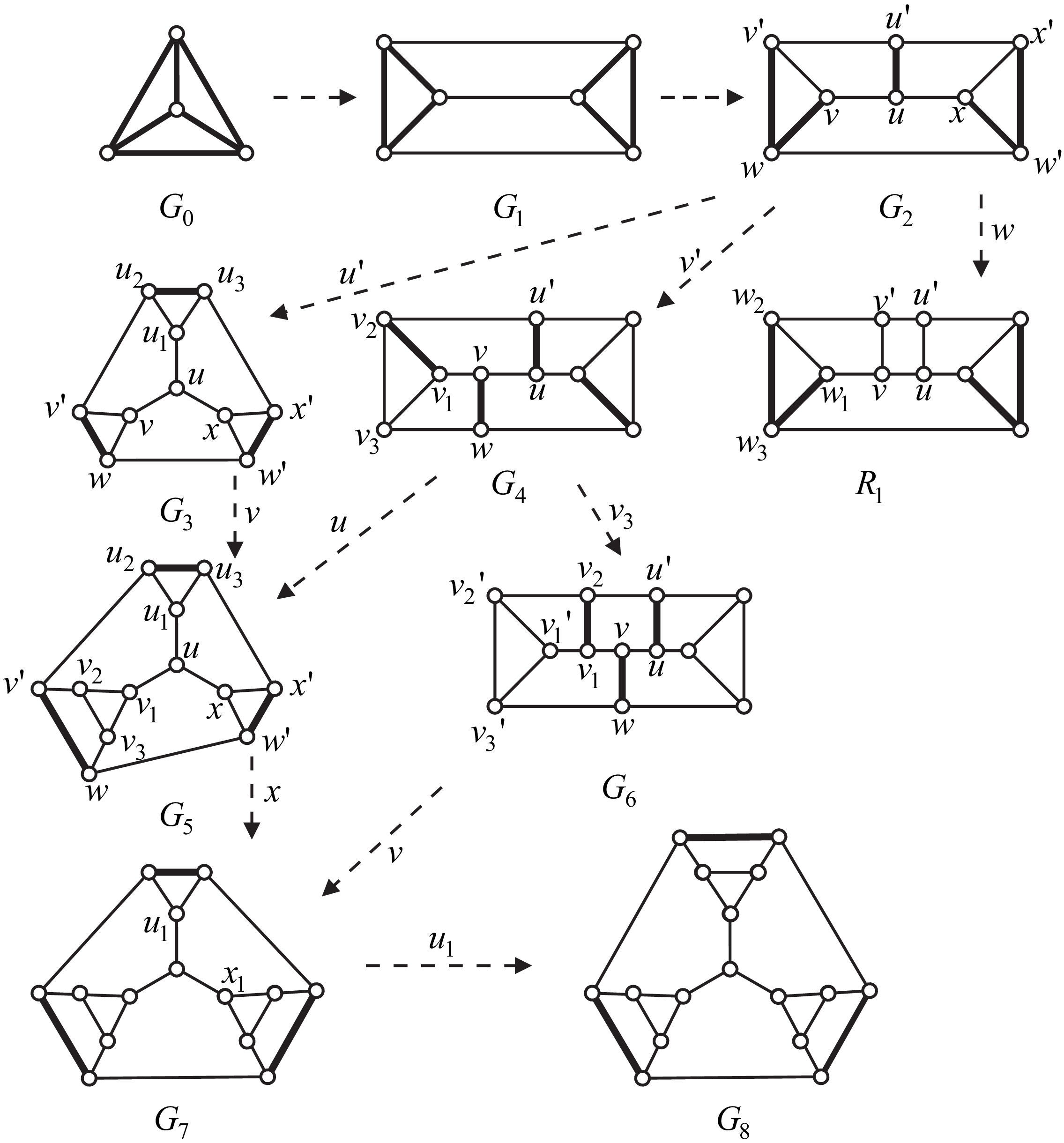}
   \caption{\label{f4}All possible bricks in Theorem \ref{th1}:
   those bold edges are their forcing edges.}
\end{figure}
\section{ Proof of Theorem \ref{th1}}

In the following, we first prove the sufficiency of Theorem \ref{th1}.

\textbf{Sufficiency.}
   As shown in Fig. \ref{f4}, it can be checked all bold edges of
   $G_i$ are forcing edges.
   For each $2\leq i\leq 8$, let
   $S_i$ be the union of those perfect matchings of $G_i$, each of which
   contains at least one forcing edge.
   Then we can check that $S_i=E(G_i)$.

   \begin{Cla}
      \label{cla6}
      Let $e$ be a forcing edge of a matching covered graph $G$, and $M$ be the unique perfect matching of $G$
      containing $e$. Then for any $e'\in M\setminus \{e\}$, $G-e'$ is not matching covered.
   \end{Cla}
   \begin{proof}
      If $G-e'-V(e)$ has a perfect matching $M'$, then $M'\cup \{e\}$ is a perfect matching of
      $G$. Since $M$ contains $e'$ but $M'\cup \{e\}$ does not, $M\neq M'\cup \{e\}$.
      Since both $M$ and $M'\cup \{e\}$ contain $e$, a contradiction
      to that $e$ is a forcing edge of $G$.
      So $e$ does not belong to any perfect matching of $G-e'$,
      and $G-e'$ is not matching covered.
   \end{proof}
   For each $2\leq i\leq 8$, let $e$ be a $b$-invariant edge of $G_i$.
   Then $e$ is removable, i.e. $G_i-e$ is matching covered graph.
   Let $F_i$ be the set of all forcing edges of $G_i$.
   By Claim \ref{cla6}, $e\in E(G_i)\setminus (S_i\setminus F_i)=(E(G_i)\setminus S_i)\cup F_i=F_i$, i.e.
   $e$ must be a forcing edge of $G_i$.
   Thus all $b$-invariant edges of $G_i$ are its forcing edges,
   and this completes the proof of sufficiency.

A \emph{5-wheel} $W_5$ is a graph formed by connecting a single vertex $x$ to each vertex of
a 5-cycle. Note that $W_5$ is a brick.

   \textbf{Necessity.} Since $G$ is a cubic brick, 
   other than $K_4$, $\overline{C}_6$, and the Petersen graph, 
   $G$ contains at least one $b$-invariant edge by Theorem \ref{ex-b}, which is forcing.
   By Theorem \ref{forcing-edges}, $G$ contains $K_4$ as a base.
   Since $G\neq K_4$ and $\overline{C_6}$, $|V(G)|\geq 8$.
   If $|V(G)|=8$, then $G$ is isomorphic to a bicorn, i.e.
   the graph $G_2$ shown in Fig. \ref{f4}.
   In the following, we label the vertices in $G_i$   as shown in Fig. \ref{f4} for
   $2\leq i\leq 8$.
   \begin{Cla}
      \label{u-v}
      If $|V(G)|=10$, then $G$ is isomorphic to $G_3$ or $G_4$.
   \end{Cla}
   \begin{proof}
   Obviously $G_2$ is a base of $G$.
   If $G=G_2^{\triangle}(w)$, then as shown in Fig. \ref{f4}, $G\cong R_1$,
   and $vv'$ and $uu'$ are two bottom edges of two   pyramids  of $R_1$ respectively.
   By Lemma \ref{lem4}, $vv'$ and $uu'$ are both $b$-invariant edges of $R_1$.
   However, neither $vv'$ nor $uu'$ is a forcing edge of $R_1$, a contradiction.
   By symmetry of $G_2$, $G$ is isomorphic to
   $G_2^{\triangle}(u')=G_3$ or $G_2^{\triangle}(v')=G_4$.
   \end{proof}
   Next we assume that $|V(G)|\geq 12$.

   \begin{Cla}
      \label{G34}
      If $|V(G)|\geq 12$, then $G$ contains  $G_3$ or $G_4$ as a base.
   \end{Cla}
   \begin{proof}
      It is obvious that at least one of $R_1$, $G_3$ and $G_4$ is a base of $G$.
      If $R_1$ is a base of $G$, then by the proof of Claim \ref{u-v},
      $uu'$ and $vv'$ are both $b$-invariant but not forcing in $R_1$.
      If those vertices $u$, $u'$, $v$ and $v'$ in $R_1$ are also vertices of $G$,
      then by Corollary \ref{bf}, $uu'$ and $vv'$ are also $b$-invariant but not forcing
      edges of $G$, a contradiction.
      So at least one of vertices in $\{u, u', v, v'\}$ are not in $G$,
      which implies that $G_3$ is a base of $G$.
      Therefore, $G$ contains $G_3$ or $G_4$ as a base.
   \end{proof}
   By Claim \ref{G34}, we will consider the following two cases.
   \begin{Cas}
      $G_3$ is a base of $G$.
   \end{Cas}
   By the definition of $R_0$
   (see Fig. \ref{f2}), $R_0\cong G_3^{\triangle}(u)$. Moreover, we have the following claim.
   \begin{Cla}
      \label{ex-Q}
      If $|V(G)|\geq 12$, then
      $G\ncong R_0$ and $G$ does not contain the graph $R_0$ as a base.
   \end{Cla}
   \begin{proof} By Proposition \ref{pro1}, $R_0$ has no forcing edges.
   So $G\ncong R_0$.
      Suppose $G$ contains $R_0$ as a base.
      By Corollary \ref{Y-Op}, $G$ contains no forcing edges.
      By Theorem \ref{ex-b}, $G$ contains at least one $b$-invariant edge, a contradiction.
   \end{proof}

   By sufficiency, all $b$-invariant edges of $G_3$ are forcing edges.
   Since $u_2u_3$, $v'w$ and $w'x'$ are all forcing edges of $G_3$,
   they are also $b$-invariant edges of $G_3$ by symmetry.
   Let $R_2=G_3^{\triangle}(v')$ (see Fig. \ref{f5}). Now we have the following claim.

   \begin{figure}[ht]
      \centering
      \includegraphics[scale=0.35]{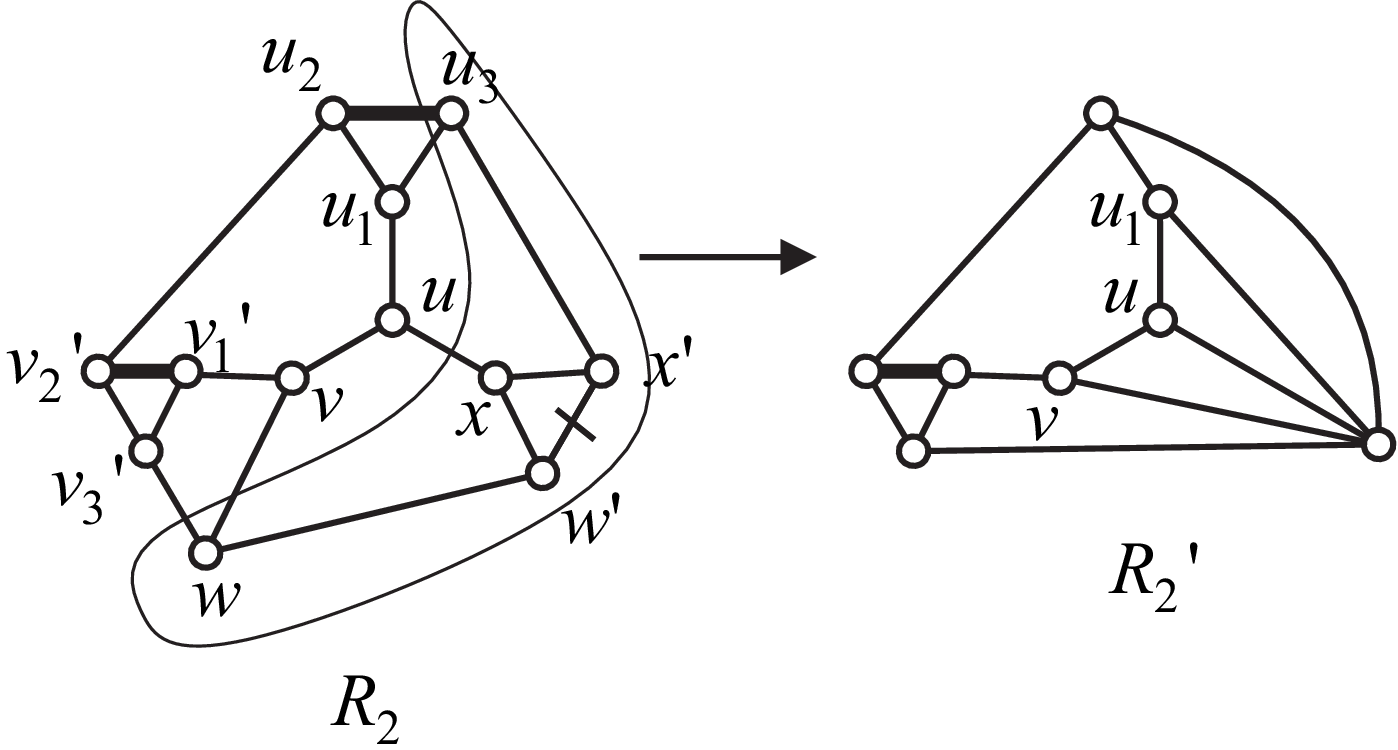}
      \caption{\label{f5}Performing a $Y\rightarrow \triangle$-operation on one end vertex of a forcing
      edge of $G_3$.}
   \end{figure}

   \begin{Cla}
      \label{ex-R}
      $w'x'$ is a $b$-invariant but not forcing edge of $R_2$.
   \end{Cla}
   \begin{proof}
      In $R_2$, since any perfect matching of the 4-cycle $vv_1'v_3'wv$ can be extended
      to a perfect matching containing $w'x'$, $w'x'$ is not a forcing edge of $R_2$.
   Next we show $x'w'$ is a $b$-invariant edge of $R_2$.
   Let $X=\{w,w',x,x',u_3\}$.
   Then the subgraph of $R_2-x'w'$ induced by $X$ is exactly an even path
   $ww'xx'u_3$. As both $w'$ and $x'$ have degree two in $R_2-x'w'$,
     any perfect matching of $R_2-x'w'$ contains exactly one  edge in $\nabla(X)$.
   So $\nabla(X)$ is a nontrivial tight cut of $R_2-x'w'$.
   Obviously $(R_2-x'w')/\overline{X}$ is a matching covered bipartite graph.
   Let $R_2'=(R_2-x'w')/X$.
   As shown in Fig. \ref{f5}, $R_2'$ contains $W_5$ as a base.
   Since $W_5$ is a brick, $R_2'$ is also a brick by Theorem \ref{brick-expand}.
   So $x'w'$ is a $b$-invariant edge of $R_2$ and we are done.
   \end{proof}

   \begin{Cla}
      \label{R-b}
      If $|V(G)|=12$, then $G\cong G_5$.
      If $|V(G)|\geq 14$, then $R_2$ is not a base of $G$.
   \end{Cla}
   \begin{proof}
      By Claims \ref{ex-Q} and \ref{ex-R}, if $|V(G)|=12$, then $G\cong G_5$.
      Next we consider $|V(G)|\geq 14$.
      Suppose, to the contrary, that $R_2$ is a base of $G$.
      For $R_2$, if we perform a $Y\rightarrow \triangle$-operation on
      $w$, $w'$, $x'$ or $u_3$,
      then similar to the process from $G_3$ to $R_2$ in Claim \ref{ex-R},
      it will deduce that the subgraph,
      which corresponds to the triangle $u_1u_2u_3u_1$ or the triangle $wvv'w$ in $G_3$,
      contains no forcing edges.
      By Corollary \ref{Y-Op}, $G$ contains at most one forcing edge.
      By Theorem \ref{ex-b} and the fact that $G_2$ is the only brick with exactly one
      $b$-invariant edge (see Theorem 1.1 (iii) in \cite{Carvalho2002II}), $G$ contains
      at least two $b$-invariant edges, a contradiction.
      Thus $w$, $w'$, $x'$ and $u_3$ are vertices of $G$.
      If $x$ is not a vertex of $G$, then $w'x'$ is the bottom edge of a pyramid in $R_2^{\triangle}(x)$.
      By Corollary \ref{bf}, $w'x'$ is a $b$-invariant, but not forcing edge of $G$,
      a contradiction again.
      So $x$ is also a vertex of $G$. Similar to the proof of Claim \ref{ex-R},
      choosing the same $X=\{w,w',x,x',u_3\}$, we can show that $\nabla(X)$ is a
      nontrivial tight cut of $G-x'w'$ and $(G-x'w')/\overline{X}$ is a
      matching covered bipartite graph.
      Since $(G-x'w')/X$ contains a brick $R_2'=(R_2-x'w')/X$ as a base,
      $(G-x'w')/X$ is also a brick by Theorem \ref{brick-expand}.
      So $w'x'$ is a $b$-invariant but not a forcing edge of $G$, a contradiction.
   \end{proof}

   \begin{Cla}
      If $|V(G)|\geq 14$, then $G$ is isomorphic to  $G_7$ or $G_8$ (see Fig. \ref{f4}).
   \end{Cla}
   \begin{proof}
      By Claims \ref{ex-Q} and \ref{ex-R}, $G_5$ is a base of $G$.
      Next we show the triangle $T=v_1v_2v_3v_1$ in $G_5$ is also a triangle of $G$.
     It should be noted that $G_5$ contains exactly three forcing edges:
     $u_2u_3$, $v'w$ and $w'x'$.

      Let $R_3=G_5^{\triangle}(v_1)$, as shown in Fig. \ref{f6}.
      $R_3-v'-w$ contains at least two different perfect matchings, which implies
      that $v'w$ is not a forcing edge of $R_3$.
      Since $v'w$ is the bottom edge of a pyramid of $G_5$,
      $v'w$ is a $b$-invariant edge of $R_3$ by Lemma \ref{lem4}.
      So $G\ncong R_3$.
      By Claim \ref{R-b}, $v'$ and $w$ are vertices of $G$.
      Further, from Corollary \ref{bf}, $R_3$ is not a base of $G$.

      \begin{figure}[h]
         \centering
         \includegraphics[scale=0.35]{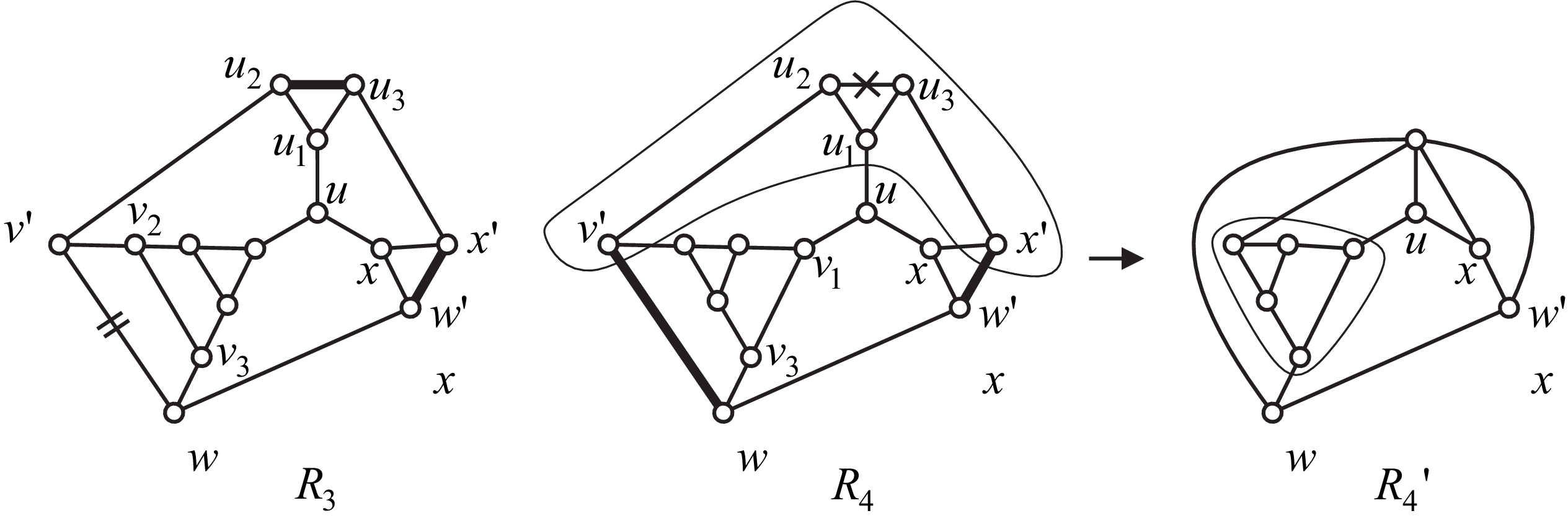}
         \caption{\label{f6}
         $R_3=G_5^{\triangle}(v_1)$, $R_4=G_5^{\triangle}(v_2)$ and $R_4'=(R_4-u_2u_3)/X$,
         where $X=\{u_1,u_2,u_3,v',x'\}$.}
      \end{figure}

      If we perform a $Y\rightarrow \triangle$-operation on $v_2$ or $v_3$ of $G_5$,
      say $v_2$,
      then $G_5^{\triangle}(v_2)\cong R_4$ (see Fig. \ref{f6}).
      Since $R_4-u_2-u_3$ contains two different perfect matchings,
      $u_2u_3$ is not a forcing edge of $R_4$.
      In the graph $R_4-u_2u_3$, let $X=\{u_1,u_2,u_3,v',x'\}$.
      Then $\nabla(X)$ is a nontrivial tight cut of $R_4-u_2u_3$.
      Obviously $(R_4-u_2u_3)/\overline{X}$ is a matching covered bipartite graph.
      Since $(R_4-u_2u_3)/X$ is isomorphic to the graph $R_4'$ in Fig. \ref{f6}, it
      contains $W_5$ as a base. Since $W_5$ is a brick,
      $(R_4-u_2u_3)/X$ is also a brick by Theorem \ref{brick-expand}.
      Thus $u_2u_3$ is a $b$-invariant but not forcing edge of $R_4$.
      Further, $G\ncong R_4$.

      Assume $R_4$ is a base of $G$.
      By Claim \ref{R-b}, $u_2$, $u_3$, $v'$  and $x'$ are vertices of $G$,
      i.e. $X\setminus \{u_1\}\subset V(G)$.
      If $u_1$ is also a vertex of $G$,
      then similar to the proof that $u_2u_3$ is a $b$-invariant edge of $R_4$,
      we can choose the same $X$ to show
      that  $\nabla(X)$ is a nontrivial tight cut of $G-u_2u_3$ and
      $(G-u_2u_3)/\overline{X}$ is a matching covered bipartite graph.
      Since $(G-u_2u_3)/X$ contains a brick $R_4'=(R_4-u_2u_3)/X$ as a base,
      $(G-u_2u_3)/X$ is also a brick by Theorem \ref{brick-expand}.
      So $u_2u_3$ is also a $b$-invariant but not forcing edge of $G$, a contradiction.
      If $u_1$ is not a vertex of $G$, then
      we can perform a $Y\rightarrow \triangle$-operation on $u_1$ for the graph $R_4$.
      Then $u_2u_3$ corresponds to the bottom edge of a   pyramid in the
      $R_4^{\triangle}(u_1)$.
      By Corollary \ref{bf}, $u_2u_3$ is also a $b$-invariant but not forcing edge of $G$,
      a contradiction.

      Since none of $R_0$, $R_2$, $R_3$ and $R_4$ is a base of $G$,
      but $G_5$ is a base of $G$ and $|V(G)|\geq 14$, $G\in \{G_7,G_8\}$.
   \end{proof}

   \begin{Cas}
      \label{cas2}
      $G_3$ is not a base of $G$.
   \end{Cas}

   By the proof of Claim \ref{G34}, $G_4$ is a base of $G$ and  $R_1$ is not a base of $G$.
   \begin{Cla}
      If $|V(G)|=12$, then $G\cong G_6$.
   \end{Cla}
   \begin{proof}
   We consider the triangle $v_1v_2v_3v_1$ in $G_4$ (see Fig. \ref{f4}).
   If $G=G_4^{\triangle}(v_1)$, then $G\cong R_5$ (see Fig. \ref{f7}).
   In $R_5$, $v_2v_3$ is the bottom edge of a pyramid.
   By Lemma \ref{lem4}, $v_2v_3$ is a $b$-invariant edge of $R_5$, a contradiction.
   Since neither $R_1$ nor $G_3$ is a base of $G$,
   $G\cong G_4^{\triangle}(v_3)\cong G_6$ by symmetry.
   \end{proof}

   \begin{figure}[ht]
      \centering
      \includegraphics[scale=0.35]{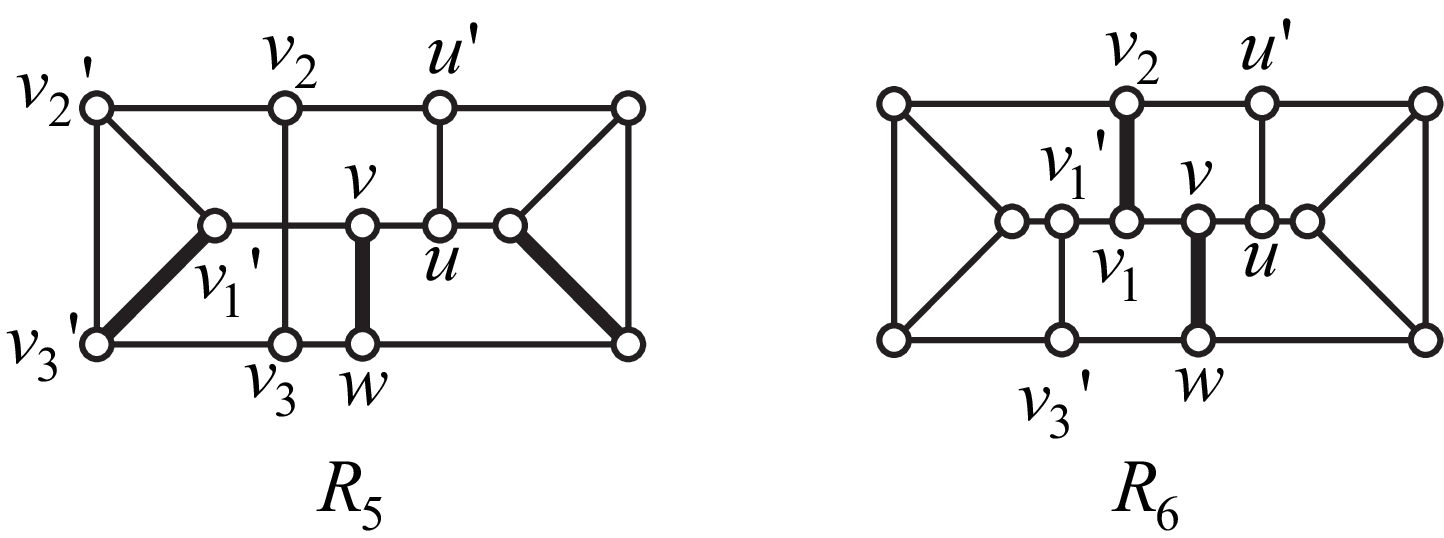}
      \caption{\label{f7}$R_5=G_4^{\triangle}(v_1)$ and $R_6=G_6^{\triangle}(v_2')$.}
   \end{figure}

   Next we show that $|V(G)|\leq 12$ in this case.
   If $|V(G)|\geq 14$, then either $R_5$ (see Fig. \ref{f7}) or $G_6$ is a base of $G$.
   If $R_5$  is a base of $G$, then according to the assumption that 
   $G_3$ is not a base of $G$, we can
   deduce that both $v_2$ and $v_3$  are also vertices of $G$.
   However, by Corollary \ref{bf},
   the edge $v_2v_3$ in $R_5$ is a $b$-invariant but not forcing edge of $G$,
   a contradiction.
   So $R_5$ is not a base of $G$.

   If $G_6$ is a base of $G$, then by the assumption, all the vertices not in any triangle of
   $G_6$ are also vertices of $G$.
   Since neither $R_1$ nor $R_5$ is a base of $G$,
   $G$ is isomorphic to $R_6$ (see Fig. \ref{f7}) or has a base $R_6$.
   For the graph $R_6$,
   obviously, $v_1'v_3'$ is the bottom edge of a pyramid.
   By Lemma \ref{lem4}, $v_1'v_3'$ is a $b$-invariant edge of $R_6$.
   Since $v_1'v_3'$ is not a forcing edge of $R_6$, $G\ncong R_6$.
   If $R_6$ is a base of $G$, then $v_1'$ and $v_3'$ in $R_6$
   are both vertices of $G$ by assumption.
   By Corollary \ref{bf}, $v_1'v_3'$ is a $b$-invariant but not forcing edge of $G$,
   a contradiction.
   Therefore, $|V(G)|\leq 12$ and $G\cong G_6$.

   According to above discussion, we know that $G\in \{G_2\ldots,G_8\}$,
   and this completes the proof of Theorem \ref{th1}.
   \\[7pt]
   \indent
   Let $v_0$ be a 2-degree vertex of a matching covered graph $G$ with $|V(G)|\geq 4$,
   and $v_1$ and $v_2$ be two neighbours of $v$.
   Then the operation of contracting $\{v_0,v_1,v_2\}$ in $G$ to a single vertex is
   called the \emph{bicontraction} of $v_0$.
   An edge $e$ of a brick $G$ is called a \emph{thin edge} if
   $\widehat{G-e}$ is a brick, where $\widehat{G-e}$ means to bicontract 2-degree vertices
   of ${G-e}$ repeatedly until there is no 2-degree vertices. Obviously,
   every thin edge is a $b$-invariant edge.

It can be checked that all $b$-invariant edges in $\{G_i,2\leq i\leq 8\}$ are also thin edges.
In the process of performing $Y\rightarrow \triangle$-operations, $R_i$ ($1\leq i\leq 6$)
in above proof has a $b$-invariant but not forcing edge.
It can be checked that all $b$-invariant edges in $\{R_i,1\leq i\leq 6\}$ are also thin edges.
It means that if all thin edges of $G$ are forcing edges, then $G$ does not belong to $\{R_i,1\leq i\leq 6\}$ and has no base isomorphic to $R_i$.
So we have the following result.

\begin{Not}
   $\{G_i, 2\leq i\leq 8\}$, $K_4$, $\overline{C_6}$ and the Petersen graph are
   all the cubic bricks satisfying that all thin edges are forcing edges.
\end{Not}

\end{document}